\documentclass[12pt]{article}
\usepackage[latin1]{inputenc}
\usepackage{amsmath}
\usepackage{amsfonts}
\usepackage{amssymb}
\usepackage[american]{babel}
\usepackage{graphicx}
\usepackage{amsthm}
\usepackage{color}
\usepackage{xcolor}
\usepackage[normalem]{ulem}
\usepackage{hyperref}

\begin{document}


\title{On $d$-dimensional nowhere-zero $r$-flows on a graph}{}

\newtheorem{theorem}{Theorem}
\newtheorem{proposition}[theorem]{Proposition}
\newtheorem{lemma}[theorem]{Lemma}
\newtheorem{definition}[theorem]{Definition}
\newtheorem{example}[theorem]{Example}
\newtheorem{corollary}[theorem]{Corollary}
\newtheorem{conjecture}[theorem]{Conjecture}
\newtheorem{remark}{Remark}
\newtheorem{problem}{Problem}

\author{Davide Mattiolo \footnote{Department of Computer Science, KU Leuven Kulak, 8500 Kortrijk, Belgium}, Giuseppe Mazzuoccolo \footnote{Dipartimento di Informatica, Universit\`{a} degli Studi di Verona, Italy},\\Jozef Rajn\'{i}k \footnote{Department of Computer Science, Comenius University in Bratislava, Slovakia}, Gloria Tabarelli \footnote{Dipartimento di Matematica, Universit\`{a} di Trento,
Italy}}
\maketitle

\begin{abstract}
A $d$-dimensional nowhere-zero $r$-flow on a graph $G$, an $(r,d)$-NZF from now on,
is a flow where the value on each edge is an element of $\mathbb{R}^d$ whose (Euclidean) norm lies in the interval $[1,r-1]$.
Such a notion is a natural generalization of the well-known concept of circular nowhere-zero $r$-flow (i.e.\ $d=1$).
For every bridgeless graph $G$, the $5$-flow Conjecture claims that $\phi_1(G)\leq 5$, while a conjecture by Jain suggests that $\phi_d(G)=1$, for all $d \geq 3$. Here, we address the problem of finding a possible upper-bound also for the remaining case $d=2$. We show that, for all bridgeless graphs, $\phi_2(G) \le 1 + \sqrt{5}$ and that the oriented $5$-cycle double cover Conjecture implies $\phi_2(G)\leq \tau^2$, where $\tau$ is the Golden Ratio. Moreover, we propose a geometric method to describe an $(r,2)$-NZF of a cubic graph in a compact way, and we apply it in some instances.
Our results and some computational evidence suggest that $\tau^2$ could be a promising upper bound for the parameter $\phi_2(G)$ for an arbitrary bridgeless graph $G$. We leave that as a relevant open problem which represents an analogous of the $5$-flow Conjecture in the $2$-dimensional case (i.e. complex case).

\end{abstract}

\section{Introduction}\label{section intro}

Let $r\geq2$ be a real number and $d$ a positive integer, a \textit{$d$-dimensional nowhere-zero $r$-flow} on a graph $G$, denoted $(r,d)$-NZF on $G$, is an orientation of $G$ together with an assignment $\varphi \colon E(G)\to \mathbb{R}^d$ such that, for all $e\in E(G),$ the (Euclidean) norm of $\varphi(e)$ lies in the interval $[1,r-1]$ and, for every vertex, the sum of the inflow and outflow is the zero vector in $\mathbb{R}^d$. 
The \emph{$d$-dimensional flow number} of a bridgeless graph $G$, denoted by $\phi_d(G)$, is defined as the infimum of the real numbers $r$ such that $G$ admits an $(r,d)$-NZF. 
Note that, by Seymour's 6-flow theorem \cite{Seymour6flow} we have that  $\phi_d(G) \leq 6$ for every $d$.
Actually, $\phi_d(G)$ is a minimum: due to the above upper bound, it suffices to consider only the set of feasible $d$-dimensional nowhere-zero $r$-flows with $r \leq 6$, which can be represented as a compact subset of $\mathbb{R}^{d\cdot |E(G)|}$, and the function that assigns to every feasible flow the maximum norm among its components, that are elements of $\mathbb{R}^d$, is continuous.

In the above definitions it is not restrictive to assume any graph $G$ to be connected. So we only consider connected graphs in the rest of the paper.

The notion of $(r,d)$-NZF includes some parameters already considered in the literature. First of all, the $1$-dimensional case, that is $\phi_1(G)$, is nothing but the classical circular flow number of a graph (see \cite{tarsizhang}). A famous conjecture by Tutte can be stated according to our notation as follows.

\begin{conjecture}[$5$-flow Conjecture]\label{conj:5flow}
Let $G$ be a bridgeless graph. Then, $\phi_1(G)\leq 5$.
\end{conjecture}

An upper bound for $\phi_d$ is also conjectured for $d\geq 3$. Indeed, Jain suggested (see \cite{DeVos}) that every bridgeless graph admits a nowhere-zero flow with flow values taken on the unitary sphere $S^2$, that is the set of unit vectors of $\mathbb{R}^3$. Clearly, such a conjecture can be stated in our terminology as follows.

\begin{conjecture}[$S^2$-flow Conjecture]\label{conj:S2}
 Let $G$ be a bridgeless graph. Then, $\phi_d(G)=2$ for every $d>2$.
\end{conjecture}

Let us remark that, along the paper, we will use the term cycle in its largest acception of subgraph with all vertices of even degree. Such a use is quite common in this context and permits to simplify the presentation. 
It is already observed in \cite{Tho} that $\phi_7(G)=2$ for every bridgeless graph $G$. This is a consequence of a covering result by Bermond, Jackson and Jaeger \cite{BJJ}, claiming that every bridgeless graph $G$ has seven cycles such that every edge of $G$ is contained in exactly four of them. 
Moreover, the Berge-Fulkerson conjecture (see \cite{fulkerson, berge}), if it holds true, implies that every bridgeless cubic graph has six cycles such that every edge is in exactly four of them. As noted in \cite{Tho}, this would imply that $\phi_6(G)=2$ for any bridgeless cubic graph $G$. In a similar way, if the conjecture of Celmins and Preissmann on the existence, for every bridgeless graph, of five cycles covering each edge twice is true, then $\phi_5(G)=2$ for any bridgeless graph $G$.

Now, it is natural to ask what is a general upper bound for the $2$-dimensional case. Indeed, Conjecture \ref{conj:5flow} and Conjecture \ref{conj:S2} do not address the case $d=2$.  As far as we know, such a question is not considered in the literature yet, and one of the main goals of this paper is proposing a general upper bound for $\phi_2(G)$, see Corollary \ref{cor:phi2bound}, Theorem \ref{thm:sqrt5bound} and Problem \ref{pro:phi2bound}.

Let us note that a $2$-dimensional nowhere-zero $r$-flow can be viewed as a generalization of a $1$-dimensional nowhere-zero flow where flow values are taken in the complex field $\mathbb{C}$. This notion is already considered in \cite{Tho} in relation with Conjecture \ref{conj:S2}. Among other results, it is proved that $\phi_2(G)=2$ if $G$ is $6$-edge-connected, but no discussion about a general upper bound for $\phi_2$ is proposed by the author.   
Some other results on 2-dimensional nowhere-zero $r$-flows are obtained in \cite{Zhangandal}, where the special case of flow values taken in the 2-dimensional unit sphere $S^1$ is considered.

\section{Possible upper bounds for $\phi_2(G)$}

A cycle double cover of a graph $G$ is a collection of cycles that together include each edge of $G$ exactly twice. Notice that a subgraph with no edges is also considered a cycle of $G$.
The existence of a cycle double cover for each bridgeless graph is a famous unsolved problem, posed by Seymour and Szekeres \cite{Seymour, Szekeres}, and known as the cycle double cover conjecture. 
There are many variations on the cycle double cover conjecture (see \cite{ZhangBook} for a comprehensive survey).

Here we will consider one of the strongest formulations, known as the \emph{oriented $5$-cycle double cover conjecture}. In order to introduce it, we need to recall some terminology. 

If $G$ is a graph and $O$ is an orientation of the edges of $G$, we denote by $O(G)$ the directed graph so obtained and, for every edge $e \in E(G)$, we denote by $O(e)$ its orientation with respect to $O$. A subgraph $H$ of $O(G)$ is a \emph{directed cycle} of $O(G)$ if for each vertex $v$ of $H$, the indegree of $v$ equals the outdegree of $v$.

The collection $\mathcal{C}=\{O_1(C_1), \dots, O_k(C_k)\}$ of directed cycles of a graph $G$ is said to be an \emph{oriented cycle double cover} of $G$ if every edge $e$ of $G$ belongs to exactly two cycles $C_i$ and $C_j$ and the directions of $O_i(C_i)$ and $O_j(C_j)$ are opposite on $e$.

If we would like to stress the number of cycles in $\mathcal{C}$ we will write that $\mathcal{C}$ is an oriented $k$-cycle double cover of $G$.

The oriented $5$-cycle double cover conjecture, which is due to Archdeacon and Jaeger \cite{Archdeacon, Jaeger}, states the following.

\begin{conjecture}[Oriented $5$-cycle double cover Conjecture]\label{conj:o5cdcc}
Each bridgeless graph $G$ has an oriented $5$-cycle double cover.
\end{conjecture}

Now, we show that if Conjecture \ref{conj:o5cdcc} holds true, then we can deduce a general upper bound for the parameter $\phi_2$.
We shall obtain such a relation by the following more general result.

\begin{theorem}\label{thm:upperboundphi2}
Let $G$ be a bridgeless graph and $k \in \{2,3,4,5\}$. If $G$ admits an oriented $k$-cycle double cover, then
\begin{itemize}
	\item $\phi_2(G) = 2$, if $k\leq 3$;
	\item $\phi_2(G) \leq 1+\sqrt{2}$, if $k=4$;
	\item $\phi_2(G) \leq \tau^2$, if $k=5$.
\end{itemize}

  where $\tau$ denotes the Golden Ratio $\frac{1+\sqrt{5}}{2}$ \footnote{To our knowledge, the Greek letter $\tau$ represented the Golden Ratio for hundreds of years, up to the early 20th century. This ancient notation is used along the paper for the sake of a better distinction from the flow number.}.
\end{theorem}
\begin{proof}
Let $\mathcal{C}=\{O_1(C_1), \dots, O_k(C_k)\} $ be an oriented $k$-cycle double cover of $G$.
We construct a $2$-dimensional flow on $G$ as follows. Choose an arbitrary orientation $O$ of $G$ and $k$ elements $p_1,\dots,p_k$ in $\mathbb{R}^2$. For every $i \in \{1, \dots, k\}$, we add a flow value equal to $p_i$ to all edges $e\in C_i$ such that $O_i(e)=O(e)$, while we add $-p_i$ to all edges $e\in C_i$ such that $O_i(e)\ne O(e)$.

Observe that this procedure generates a $2$-dimensional flow, where every edge which belongs to $C_i \cap C_j$ receives one of the two vectors $\pm (p_i-p_j)$.
In order to obtain an $(r,2)$-NZF, we need the norm of each flow value $p_i-p_j$ to be at least one. Then, we choose $p_1,\ldots, p_k$ pointing at the $k$ vertices of a regular $k$-gon of side length $1$.
If $k \in \{2,3\}$, since $|p_i-p_j|=1$ for every $i\neq j$, then $\phi_2(G) =  2$. 
If $k=4$, since $|p_i-p_j|$ for every $i\neq j$ is either $1$ or $\sqrt{2}$, then $\phi_2(G) \leq 1+\sqrt{2}$. Finally, if $k=5$, the diagonals of a regular pentagon are in the golden ratio to its sides. Hence $|p_i-p_j|$ is equal to either $1$ or $\tau$ for every $i\neq j$, then $\phi_2(G) \leq \tau^2 (=1+\tau)$.
\end{proof}

Note that our choice of the vectors $p_1,\ldots,p_k$ in each of the three cases of the proof of Theorem \ref{thm:upperboundphi2} is known to be optimal in order to minimize the ratio between the maximum and the minimum length of $k$ points in the Euclidean plane (see \cite{BatErd}).

\begin{corollary}\label{cor:phi2bound}
The oriented $5$-cycle double cover conjecture (Conjecture \ref{conj:o5cdcc}) implies $\phi_2(G) \leq \tau^2$ for every bridgeless graph $G$. 
\end{corollary}

In Section \ref{sec:cubic}, we will discuss the problem of finding a graph $G$ such that $\phi_2(G)$ is close to $\tau^2$.

The upper bound of $\tau^2$ is obtained by assuming true a well-known conjecture. Now, we complete this section by proving a general upper bound for $\phi_2(G)$ as a consequence of the proof of the $6$-flow theorem of Seymour.
\begin{theorem}
\label{thm:sqrt5bound}
If $G$ is a bridgeless graph, then $\phi_2(G) \le 1 + \sqrt{5}$.
\end{theorem}

\begin{proof}
In the proof of the $6$-flow theorem \cite[p.~133]{Seymour6flow} Seymour showed that each bridgeless graph $G$ has an integer $2$-flow $\varphi_2$ and an integer $3$-flow $\varphi_3$ such that $\varphi_2(e) \ne 0$ or $\varphi_3(e) \ne 0$ for each edge $e \in E(G)$. Let $\varphi$ be a $2$-dimensional flow on $O(G)$, for an arbitrary orientation $O$,  such that $\varphi(e) = (\varphi_2(e), \varphi_3(e))$ for each $e \in E(O(G))$. Since $\varphi_2$ and $\varphi_3$ are $2$-flow and $3$-flow, respectively, we have $\sqrt{\varphi_2(e)^2 + \varphi_3(e)^2} \le \sqrt{1^2 + 2^2} = \sqrt{5}$. Also, one of the values $\varphi_2(e)$ and $\varphi_3(e)$ is nonzero, so $\sqrt{\varphi_2(e)^2 + \varphi_3(e)^2} \geq 1$. Thus, $\varphi$ is indeed a $(1 + \sqrt{5},2)$-NZF of $G$.
\end{proof}

\section{$2$-dimensional flows on cubic graphs} \label{sec:cubic}

In the case of nowhere-zero circular flows (i.e.\ $1$-dimensional flows) it is well known that every bridgeless graph has a nowhere-zero $r$-flow if and only if every bridgeless cubic graph has a nowhere-zero $r$-flow.
Following the same proof, one can get the following result.

\begin{proposition}\label{prop:equivalence_cubic}
	For all positive integers $d$ and real numbers $r\geq 2$, the following statements are equivalent:
	\begin{itemize}
		\item every bridgeless graph has a $d$-dimensional nowhere-zero $r$-flow;
		\item every bridgeless cubic graph has a $d$-dimensional nowhere-zero $r$-flow.
	\end{itemize}
\end{proposition}

By Proposition \ref{prop:equivalence_cubic}, there is a fixed constant $k$ such that $\phi_2(G)\le k$ for all bridgeless graphs $G$ if and only if the same holds for all bridgeless cubic graphs.

Recall that Thomassen \cite{Tho} proved that a cubic graph is bipartite if and only if it has an $S^1$-flow, that is a $2$-dimensional nowhere-zero $2$-flow. In particular, up to a rotation, one can assume that the flow values are the three cube roots of unity, that is the complex solutions of the equation $z^3=1.$

As a further step in studying the $2$-dimensional flow numbers of cubic graphs we consider those being $3$-edge-colourable. Observe that any 3-edge-colourable cubic graph has an oriented 4-cycle double cover (see for instance \cite{ZhangBook}), hence the following proposition follows from Theorem \ref{thm:upperboundphi2}.

\begin{proposition}
\label{prop:colourable}
	Let $G$ be a 3-edge-colourable cubic graph. Then $\phi_2(G)\le 1+\sqrt{2}$.
\end{proposition}

The above inequality is the best possible as one can directly check that $\phi_2(K_4)=1+\sqrt{2}$. However, we can obtain it as a
special case (i.e. $n=3$) of the following more general result which gives an exact value for $\phi_2(W_n)$, where $W_n$ is the wheel graph of order $n+1$.

The proof of Theorem \ref{thm:wheels} is quite long and technical and it will appear in another paper \cite{wheels}. 

\begin{theorem} \label{thm:wheels}
Let $W_n$ be the wheel graph of order $n+1$, for $n \geq 3$. Then $$\phi_2(W_n) = \begin{cases}
      2& \text{ if } n \text{ is even}, \\
      1 + 2\sin(\frac{\pi}{6}\cdot \frac{n}{n-1})& \text{ if } n\equiv 1, 3 \mod 6\\
      1 + 2\sin(\frac{\pi}{6}\cdot \frac{n+1}{n})& \text{ if } n\equiv 5 \mod 6. \\                                                                                                                               \end{cases}
 $$ 
\end{theorem}

The next lemma is an immediate consequence of Theorem \ref{thm:wheels}. 
\begin{lemma}
\label{lem:odd_gith}
Let $G$ be a cubic graph containing a chordless cycle $C$ of length $k$. Then $\phi_2(G)\geq \phi_2(W_k)$.
\end{lemma}
\begin{proof}
Suppose to the contrary that $\phi_2(G) <\phi_2(W_k)$. Then $G$ has an $(r,2)$-NZF $\varphi$ with $r < \phi_2(W_k)$. Contract all the vertices of $G$ that are not in $C$ to a unique vertex $v$. The obtained graph is $W_k$ and $\varphi$ induces on $W_k$ an $(r',2)$-NZF with $r'\leq r<\phi_2(W_k)$, a contradiction.
\end{proof}

Lemma \ref{lem:odd_gith} combined with the values given in Theorem \ref{thm:wheels}, gives the following.

\begin{corollary}
\label{coro:odd_gith}
Let $G$ be a cubic graph with odd-girth equal to $g$. Then $\phi_2(G)\geq \phi_2(W_g)$.
\end{corollary}

Using Corollary \ref{coro:odd_gith} we prove the following

\begin{proposition}
\label{prop:prism}
Let $n$ be odd and let $P_n$ be the prism graph of order $2n$. Then $\phi_2(P_n)=\phi_2(W_n)$.
\end{proposition}  
\begin{proof}
Since $P_n$ has odd girth $n$, we have $\phi_2(P_n) \ge \phi_2(W_n)$.

Moreover, each flow on $W_n$ can be easily extended to a flow on $P_n$ using the same vectors: for every $4$-cycle $uu'v'v$ where $uu'$ and $vv'$ are spokes of $P_n$, we set the flow from $u$ to $v$ to be the same as the flow from $v'$ to $u'$.
Thus we have $\phi_2(P_n) = \phi_2(W_n)$. 
\end{proof}

Also, Corollary \ref{coro:odd_gith}, together with Proposition \ref{prop:colourable}, implies the following result.

\begin{proposition}
\label{prop:upper_colorable}
Let $G$ be a  $3$-edge-colourable cubic graph with a triangle. Then $\phi_2(G) = 1 + \sqrt{2}$.
\end{proposition}  

Up to now, the unique infinite classes of non-bipartite cubic graphs for which we are able to determine the exact value of $\phi_2$ are the ones considered in Proposition \ref{prop:prism} and Proposition \ref{prop:upper_colorable}.

In the rest of the paper we provide upper bounds on the $2$-dimensional flow number of certain cubic graphs. To make our descriptions of $2$-dimensional flows more compact, we show that they can be equivalently represented in a geometric way. The main idea of this approach is that, by the Kirchhoff's law, the three vectors assigned to three edges incident with the same vertex correspond to a triangle in the Euclidean plane. Thus we can represent a $2$-dimensional flow as a suitable collection of triangles.

By a \emph{triangle} we mean a subset of the Euclidean plane consisting of its three sides and interior points. Let $s_1$ and $s_2$ be sides of triangles $T_1$ and $T_2$, respectively. We say that $s_1$ and $s_2$ are \emph{attachable} if
we can translate $T_1$ to $T_1'$ in such a way that the image of $s_1$ coincides with $s_2$ and $T_1'$ and $T_2$ have no common internal points. In other words, attachable sides need to be parallel, of the same length and they need to have their triangles on mutually opposite sides.
An \emph{$r$-flow triangulation} of a bridgeless cubic graph $G$ is a collection $\mathcal{T}$ containing a triangle $T_v$ for each vertex $v$ of $G$ such that
\begin{itemize}
\item[(i)] for each $v \in V(G)$, each edge incident to $v$ corresponds to a unique side of $T_v$;
\item[(ii)] lengths of sides of all triangles from $\mathcal{T}$ are from the interval $[1, r - 1]$;
\item[(iii)] for each edge $uv \in E(G)$, the sides of the triangles $T_u$ and $T_v$ corresponding to $uv$ are attachable.
\end{itemize}
\begin{proposition}
\label{prop:triangulation}
Let $G$ be a bridgeless cubic graph. Then $G$ has an $r$-flow triangulation if and only if $G$ has an $(r, 2)$-flow.
\end{proposition}
\begin{proof}
We start with the only if part. Let $O$ be an arbitrary orientation of the edges of $G$. We construct a $2$-dimensional flow on $G$ as follows. Consider an oriented edge $uv$ of $O(G)$ and let $a$ and $b$ be the vectors corresponding to the attachable sides of triangles $T_u$ and $T_v$, respectively, which are oriented in such a way that $T_u$ is on the right side of $a$ and $T_v$ is on the left side of $b$. Due to the definition of attachable sides, the vectors $a$ and $b$ have the same direction, so they are equal. We set to $a$ the flow value of the edge $uv$. Note that if we orient $uv$ in the opposite direction, it receives the opposite vector, thus we do not need any specific orientation of $G$.

We prove that this assignment is an $(r, 2)$-NZF. Consider a vertex $v$ and orient all three edges incident with $v$ as incoming.
The vectors assigned to these edges form a triangle $T_v$ and since all of them have $T_v$ on the left side, they sum up to zero.

Now for the if part, assume that $G$ has an $(r, 2)$-NZF.
For each vertex $v$ of $G$, let $e_1$, $e_2$ and $e_3$ be the oriented edges of $O(G)$ incident with $v$. For each $i \in \{1, 2, 3\}$, let $a_i$ be the flow value of $e_i$, if $v$ is the tail of $e_i$, and let $a_i$ be the opposite of flow value of $e_i$ otherwise. Then, the vectors $a_1$, $a_2$ and $a_3$ sum up to zero. Moreover, we can arrange them to form an oriented triangle $T_v$ that is on the left side of each of $a_1$, $a_2$ and $a_3$.

We prove that the triangles $T_v$ for each $v \in V(G)$ form an $r$-flow triangulation. Properties (i) and (ii) are trivially satisfied. Let $uv$ be an oriented edge of $O(G)$ with flow value $a$. Since $u$ is the tail and $v$ is the head of $uv$, the triangle $T_u$ lies on the left side of $a$ and $T_v$ lies on the right. Thus the sides of $T_u$ and $T_v$ corresponding to $a$ are attachable. Hence Property (iii) also holds.
\end{proof}

For a bridgeless cubic graph $G$, finding the representation of a $2$-dimensional flow through a flow triangulation is, in general, only a reformulation of the original problem. However, in the following examples we present flow triangulations in some ``nice'' way.
The term nice can be understood in several ways, but perhaps the most basic one requires that the 
intersection of every two different triangles $T_1$ and $T_2$, if not empty, consists either of one vertex, 
or of two coinciding sides $s_1$ and $s_2$ of $T_1$ and $T_2$, respectively. In the latter case, $s_1$ and $s_2$ correspond to the same edge of $G$ and the set of all such edges induces a connected spanning subgraph of $G$.
Examples of such nice flow triangulations of $K_4$ and $K_{3,3}$ are depicted in Figure \ref{fig:k4_k33_flow}. In all our figures, the graph is grey with its vertices placed in their corresponding triangles. Bold sides are of length $1$ and dashed ones are always the sides with maximum length.
Nevertheless, we do not know whether such a ``nice'' flow triangulation exists for every $2$-dimensional flow. 

\begin{figure}[h]
\centering
	\includegraphics[scale=0.6]{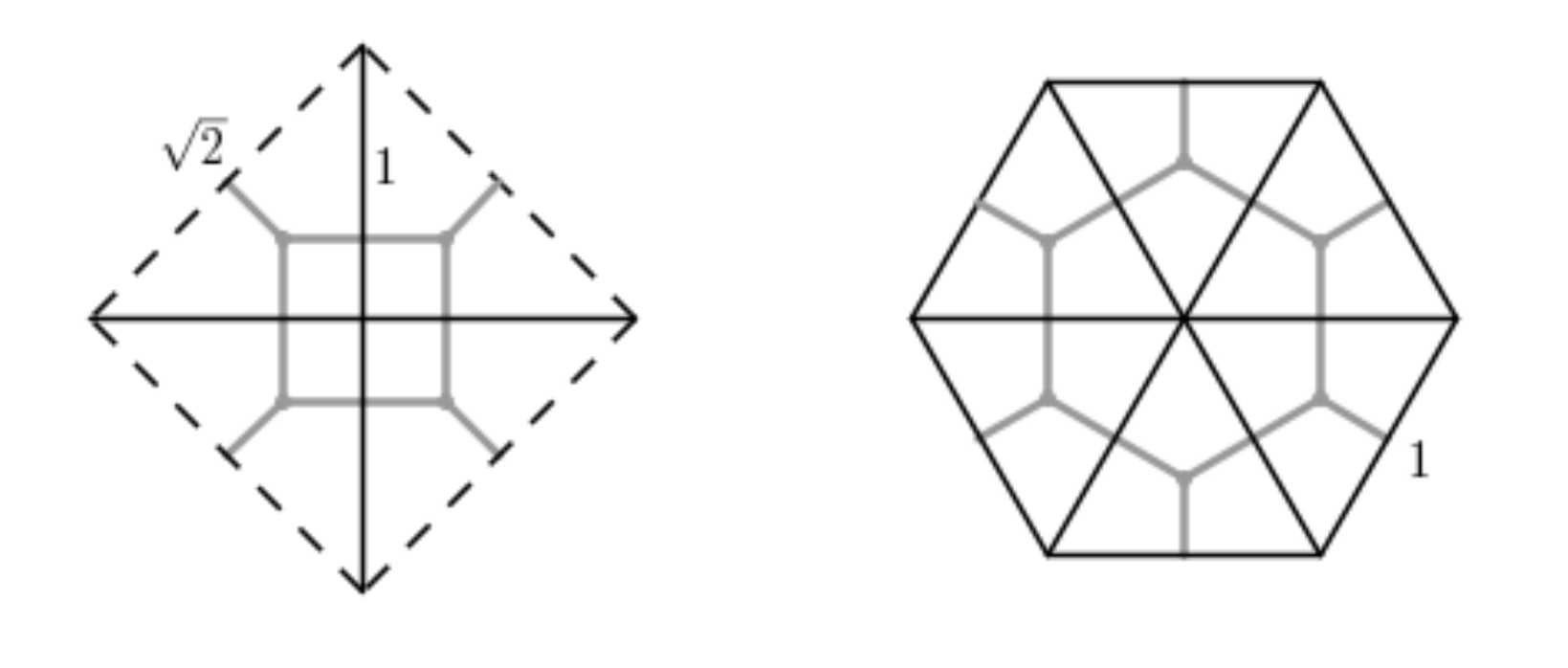}
	\caption{A $2$-flow triangulation of $K_4$ (left) and $K_{3,3}$ (right).}\label{fig:k4_k33_flow}
\end{figure}

We have already seen an upper bound on the $2$-dimensional flow number of $3$-edge-colourable cubic graphs in Proposition \ref{prop:colourable}.
As usual, in order to prove a general bound on $\phi_2(G)$ for every bridgeless cubic graph $G$, the hard case is when $G$ is not $3$-edge-colourable. Therefore, we are naturally interested in the $2$-dimensional flow number of the Petersen graph, which is the smallest such graph. Let us say that determining this value appears to be a hard problem. Here we propose an upper bound by constructing a suitable flow triangulation.

\begin{proposition}
The $2$-dimensional flow number of the Petersen graph is at most $1+\sqrt{7/3}$.
\end{proposition}
\begin{proof}
Throughout this proof, we take all the indices modulo $3$. Consider, in the real Euclidean plane, an equilateral triangle $p_1p_2p_3$ with side length $1$. For $i \in \{1, 2, 3\}$, let $q_ip_i$ be the reflection of $p_{i-1}p_i$ through $p_i$ and let $q_1'$, $q_2'$ and $q_3'$ be the points such that $q_1q_3'q_2q_1'q_3q_2'$ is a regular hexagon. By adding the segments $q_i'p_{i+1}$ and $q_i'p_{i+2}$, for each $i \in \{1, 2, 3\}$, we obtain $10$ triangles as depicted in Figure \ref{fig:petersen_flow}. The solid, dash-dotted and dashed lines have lengths $1$, $\sqrt{4/3}$ and $\sqrt{7/3}$, respectively. It is easy to check that these triangles form a $(1+\sqrt{7/3})$-flow triangulation of the Petersen graph.
\end{proof}

\begin{figure}[h]
	\centering
	\includegraphics[scale=0.25]{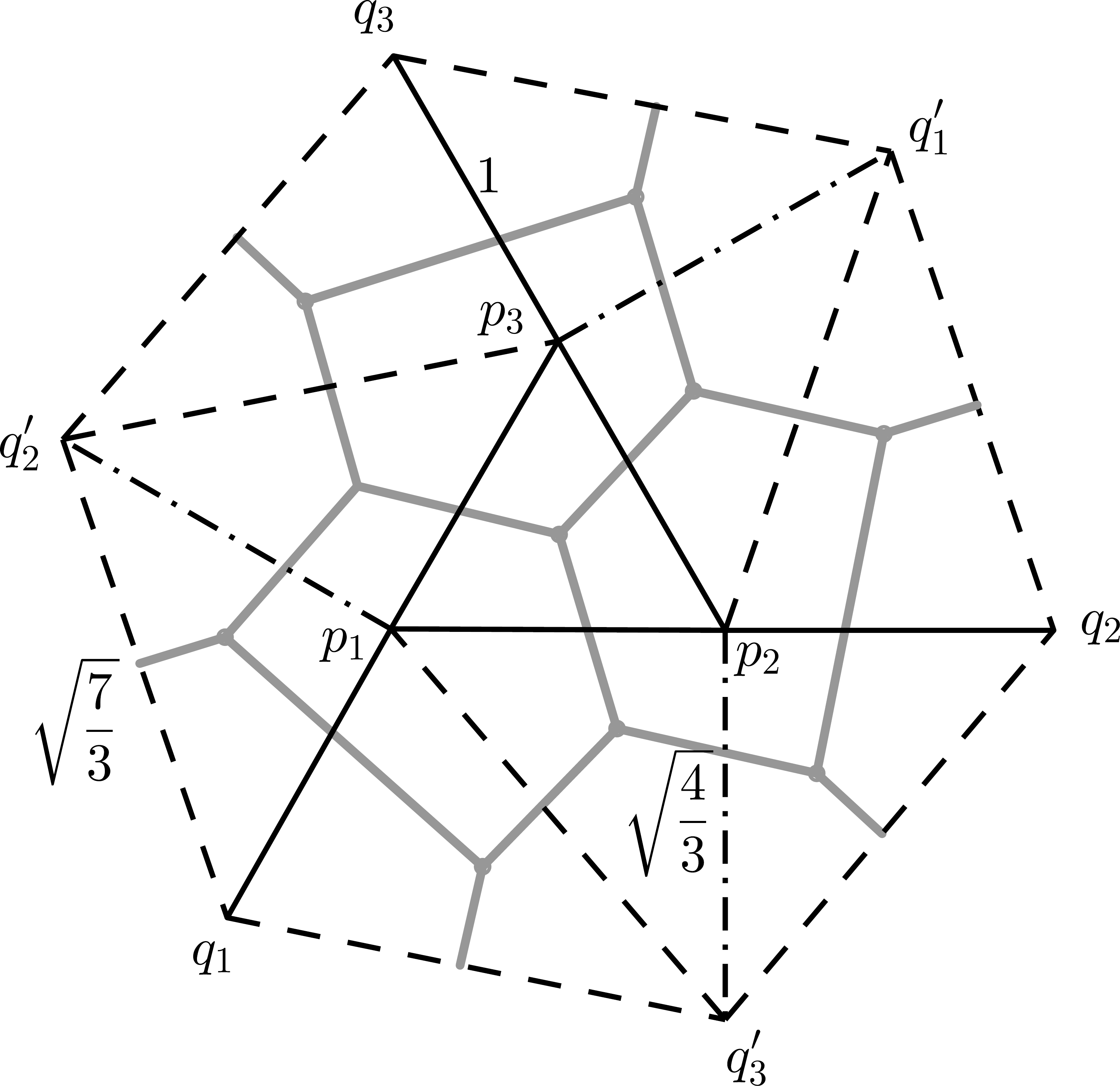}
	\caption{A $(1+\sqrt{7/3})$-flow triangulation of the Petersen graph.}\label{fig:petersen_flow}
\end{figure}

Supported by computational results we believe that this is the exact $2$-dimensional flow number of the Petersen graph. Since we currently know no tools for proving such high lower bounds on $2$-dimensional flow numbers, we propose the following conjecture.

\begin{conjecture}
The $2$-dimensional flow number of the Petersen graph is $1+\sqrt{7/3}$.
\end{conjecture}

The $1$-dimensional flow number can distinguish $3$-edge-colourable cubic graphs, which have $1$-dimensional flow number at most $4$, from the non-$3$-edge-colourable bridgeless cubic graphs having $1$-dimensional flow number greater than $4$ (see for instance \cite{Tutte5FC}). However, the $2$-dimensional flow number does not serve for this purpose. One of the counterexamples is the Isaacs snark $J_5$ (see Figure~\ref{fig:j5}) for which we show that $\phi_2(J_5) < 1 + \sqrt{2} = \phi_2(K_4)$. We have found an $(r, 2)$-NZF of $J_5$ for $r = 1 + 1.387893647$ with the help of a computer.

\begin{proposition}
$\phi_2(J_5) \le 2.387893647 < 1+\sqrt{2}$.
\end{proposition}

\begin{figure}[h]
\begin{minipage}[l]{0.45\linewidth}
\includegraphics[width=\linewidth]{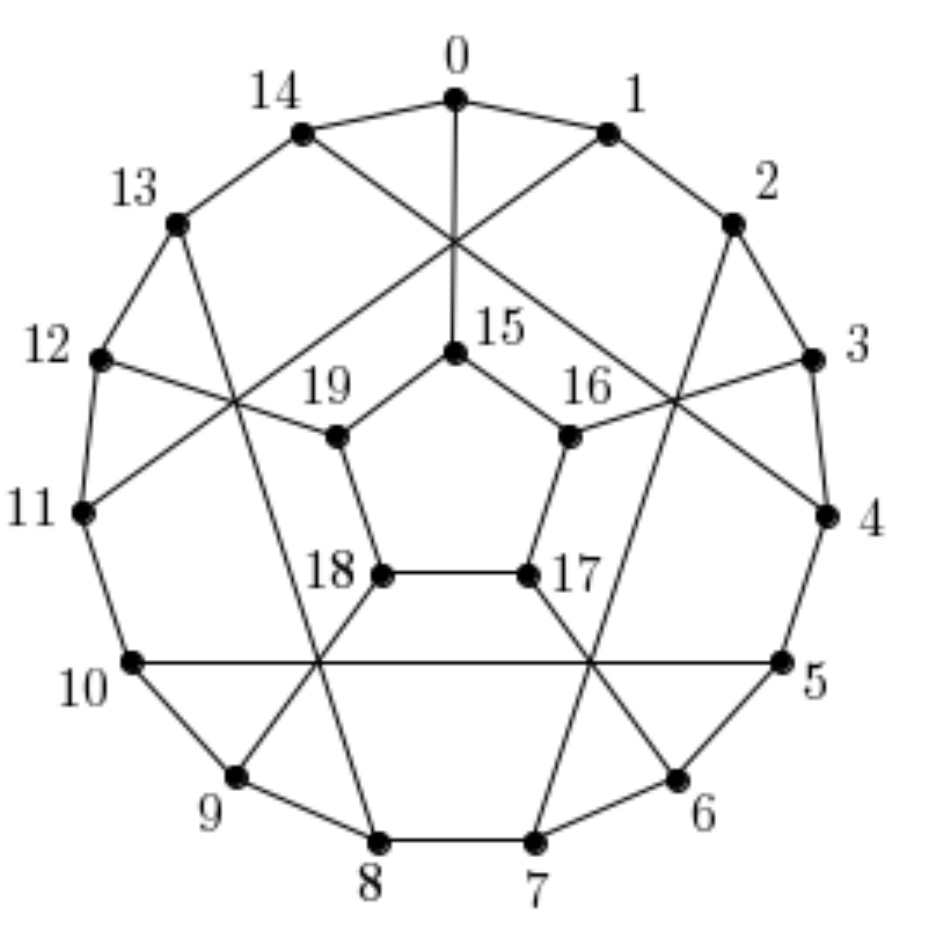}
\caption{Isaacs snark $J_5$}
\label{fig:j5}
\end{minipage}
\hfill
\begin{minipage}[l]{0.45\linewidth}
\includegraphics[width=\linewidth]{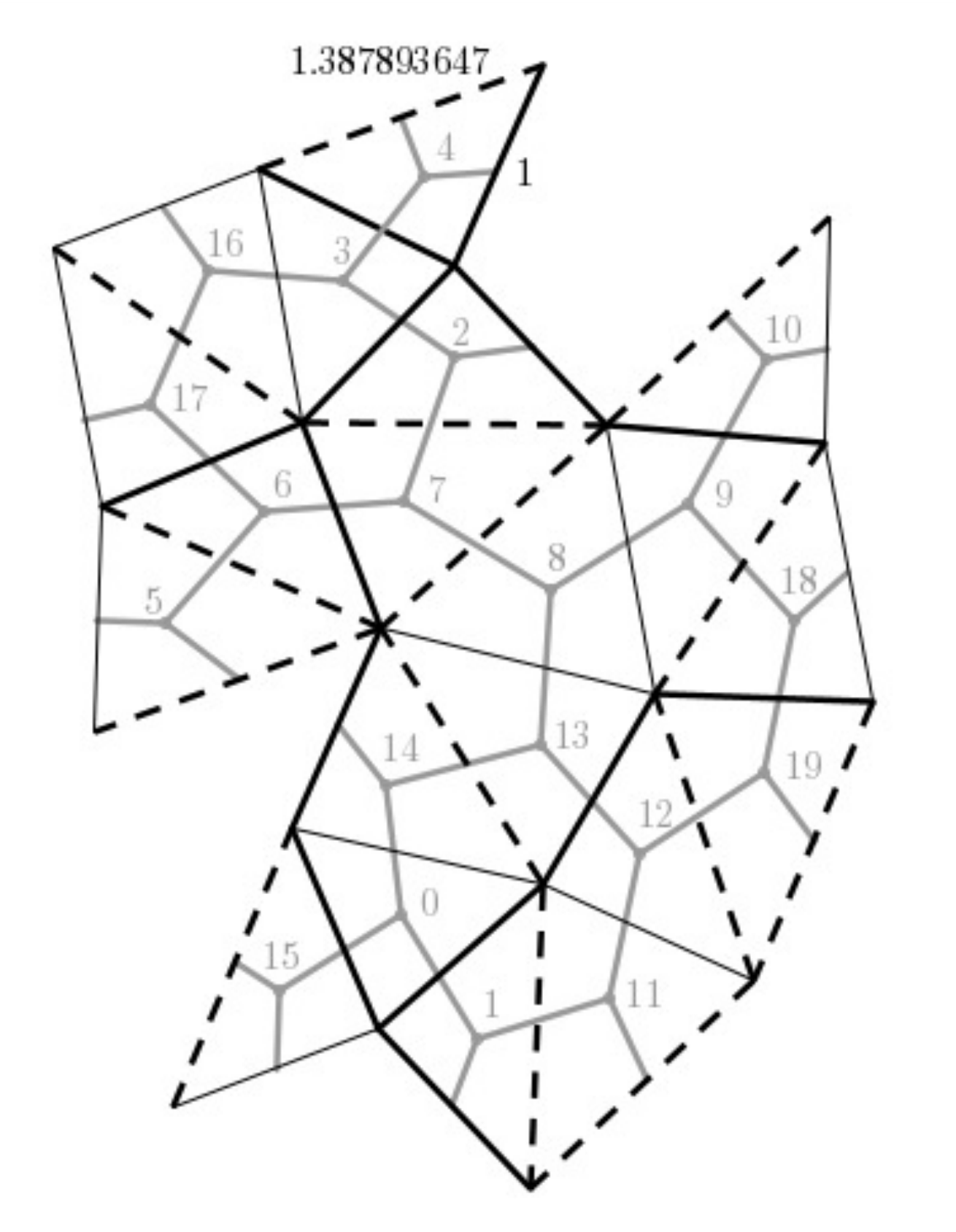}
\caption{A $2$-dimensional flow on $J_5$}
\label{fig:j5_flow}
\end{minipage}%
\end{figure}

Figure \ref{fig:j5_flow} depicts an approximation of the flow triangulation corresponding to the found flow. We emphasise only the sides with minimum (bold) and maximum (dashed) length.

The Petersen graph is the worst case for many other problems in this area. Surprisingly, this seems not to be the case here.
Indeed if we replace every vertex of $P$ with a triangle, denoting the resulting graph by $P_\Delta$, we are not able to extend our $(1 + \sqrt{7/3}, 2)$-NZF on $P$ to a $(1 + \sqrt{7/3}, 2)$-NZF on $P_\Delta$.
The best $(r,2)$-NZF flow on $P_\Delta$ we have up to now is for $r \approx 2.59$, also found by a computer.

We wonder if $\tau^2 \approx 2.618$ is the upper bound on the $2$-dimensional flow number of all bridgeless graphs and also whether this bound is reached by some graph. Therefore, we propose the following problems.

\begin{problem}\label{pro:phi2bound}
Determine if $\phi_2(G) \le \tau^2$ for every bridgeless graph $G$.
\end{problem}

\begin{problem}
Establish the existence (or not) of a bridgeless cubic graph $G$ with $\phi_2(G) = \tau^2$.
\end{problem}

We would also like to note that flow triangulations can be represented in a topological way. For instance, the $(1 + \sqrt{7/3})$-flow triangulation of the Petersen graph can be described as a dual of $P$ embedded on a torus. Similarly, the aforementioned flow triangulations for $K_4$, $K_{3,3}$ and $J_5$ can be also obtained from embeddings on some  orientable surfaces. However, since we need to measure Euclidean distance, we avoid mentioning other surfaces, where the notion of distance is not clear. 

Also, we noted that it is not clear if every $2$-dimensional flow on a cubic graph can be represented through a nice flow triangulation. We do not know the answer even for bipartite cubic graphs, which are perhaps the most simple family of cubic graphs for this problem, since each $2$-flow triangulation consists of equilateral triangles with side length $1$. Therefore, we leave it as a further open problem.

\section{Acknowledgments}

The first author is supported by a Postdoctoral Fellowship of the Research Foundation Flanders (FWO), project number 1268323N.
The third author is suppoered by the research grants APVV-19-0308, VEGA 1/0743/21 and VEGA 1/0727/22.

\end{document}